\documentclass[12pt,leqno,english]{amsart}%
\usepackage{amsfonts,amsmath,latexsym,verbatim,amscd,mathrsfs,color,array}
\usepackage[colorlinks=true,citecolor=blue]{hyperref}
\usepackage{amsmath,amssymb,amsthm,amsfonts,graphicx,color}
\usepackage{amssymb}
\usepackage{amsbsy}
\usepackage{epstopdf}
\usepackage{amsmath}
\usepackage{amsfonts}
\usepackage{graphicx}%
\providecommand{\U}[1]{\protect\rule{.1in}{.1in}}
\newtheorem{theorem}{Theorem}

\newtheorem{corollary}[theorem]{Corollary}

\newtheorem{lemma}[theorem]{Lemma}

\newtheorem{proposition}[theorem]{Proposition}
\newtheorem{remark}[theorem]{Remark}

\def\R{\mathbb{R}}
\def\N{\mathbb{N}}

\newcounter{mnotecount}[section]

\newcommand{\rmnote}[1]{}

\begin{document}

\title[Half-space theorems]{Half-space theorems for the Allen-Cahn equation and related problems}

\author[F. Hamel]{Fran{\c{c}}ois Hamel}
\address{\noindent F. Hamel: Aix Marseille Univ, CNRS, Centrale Marseille, I2M, Marseille, France}
\email{francois.hamel@univ-amu.fr}

\author[Y. Liu]{Yong Liu}
\address{\noindent Y. Liu: Department of Mathematics, University of Science and Technology of
China, Hefei, China}
\email{yliumath@ustc.edu.cn}

\author[P. Sicbaldi]{Pieralberto Sicbaldi}
\address{P. Sicbaldi: Universidad de Granada, Departamento de Geometr{\'{\i}}a y Topolog{\'{\i}}a, Facultad de Ciencias, Campus Fuentenueva, 18071 Granada, Spain \& Aix Marseille Univ, CNRS, Centrale Marseille, I2M, Marseille, France}
\email{pieralberto@ugr.es}

\author[K. Wang]{Kelei Wang}
\address{\noindent K. Wang: School of Mathematics and Statistics, Wuhan University, Wuhan,
Hubei, China}
\email{wangkelei@whu.edu.cn}
\author[J. Wei]{Juncheng Wei}
\address{\noindent J. Wei: Department of Mathematics, University of British Columbia,
Vancouver, BC V6T 1Z2, Canada }
\email{jcwei@math.ubc.ca}
\maketitle

\begin{abstract}
In this paper we obtain rigidity results for a non-constant entire solution $u$ of the Allen-Cahn equation in $\mathbb{R}^n$, whose level set $\{u=0\}$ is contained
in a half-space. If $n\leq 3$ we prove that the solution must be one-dimensional. In dimension $n\geq 4$, we prove that either the solution is one-dimensional or stays below a one-dimensional solution and converges to it after suitable translations. Some generalizations to one phase free boun\-dary problems are also obtained.
\end{abstract}


\section{Introduction and main results}\label{sec1}

We are interested in rigidity results for classical entire solutions of the
Allen-Cahn equation
\begin{equation}
-\Delta u=u-u^{3},\ \text{ }x=(x_{1},\cdots,x_{n})  \in\mathbb{R}^{n}.
\label{1}%
\end{equation}
In the simplest case $n=1,$ equation~(\ref{1}) reduces to an ODE and has a
heteroclinic solution
$$H(x)  =\tanh\left(  \frac{x}{\sqrt{2}}\right).$$
Phase plane analysis tells us that up to a translation, $H$ is
the unique monotone increasing solution in $\mathbb{R}.$ The one-dimensional solution $H$ actually plays an important role in the
theory of Allen-Cahn equation in general dimensions. Indeed, De Giorgi
\cite{DeG} conjectured that for $n\leq8,$ if a bounded solution $u$ to~\eqref{1}
is strictly monotone in one direction, say~$x_n$, then it must be
one-dimensional, which then means that $u$ is identically equal to~$H(x\cdot e+a)$ for some unit vector $e$, with $e_n>0$, and some real number $a$. This conjecture has been proved to be true for $n=2$
(Berestycki, Caffarelli and Nirenberg~\cite{BCN}, Ghoussoub and Gui~\cite{G}), $n=3$ (Ambrosio and Cabr\'e~\cite{C}), and for $4\leq
n\leq8$ (Savin~\cite{S}) under the additional limiting condition
\[
\lim_{x_{n}\rightarrow\pm\infty}u(x',x_{n})  =\pm1
\]
pointwise in $x'=(x_1,\cdots,x_{n-1})\in\R^{n-1}$. This condition implies that the level sets $\{x\in\R^n:u(x)=\mu\}$, for every $\mu\in(-1,1)$, of the function $u$ are entire graphs with respect to the first $n-1$ variables. On the other hand, for $n\geq9,$ Del Pino, Kowalczyk and the fifth author
\cite{M2} constructed monotone solutions which are not one-dimensional, showing that the condition $n\leq 8$ in the De Giorgi conjecture cannot be relaxed.

The De Giorgi conjecture can be regarded as a rigidity result for the
Allen-Cahn equation. The second rigidity result we would like to mention here
is about the classification of the solutions which are global minimizers of
the associated energy functional. Savin~\cite{S} proved that global minimizers are
one-dimensional up to dimensions $n\leq7$. The second, fourth and fifth authors constructed in
\cite{LWW} counterexamples in dimension $8$, i.e. global minimizers which are not
one-dimensional. For related rigidity results for the solutions to the Allen-Cahn equation,
we refer to~\cite{AC, C3, Farina2, Farina} and the references therein.

\medskip

In this paper, we are interested in rigidity results for the Allen-Cahn equation when the zero level set
$$\Gamma_0:=\{u=0\}=\big\{x\in\R^n:u(x)=0\big\}$$
of the solution (always understood in the classical sense) is contained in a half-space, say $\{x_n>0\}:=\{x\in\R^n:x_n>0\}$ up to translation and rotation. However, we point out that we make no assumption on the monotonicity of $u$ in a direction nor on its stability or minimizing properties.

Our first result is the following half-space rigidity result:

\begin{theorem}
\label{main}{\rm{(Weak half-space theorem)}} Let $n\leq3$ and $u$ be a non-constant
solution of~\eqref{1}. Suppose that the zero level set $\left\{
u=0\right\}  $ is contained in $\left\{  x_{n}%
>0\right\}  $. Then $u$ is one-dimensional. More precisely, there exists $a\in\R$ such that either $u(x)=H(-x_n+a)$ for all $x\in\R^n$, or $u(x)=H(x_n+a)$ for all $x\in\R^n$.
\end{theorem}

Note that, in any dimension $n\ge 1$, if $u$ is a solution of~\eqref{1}, then $u$ is necessarily bounded and
$$-1\le u\le 1\ \hbox{ in }\R^n,$$
as follows from~\cite[Proposition~1.9]{Farina98}. Furthermore, if $\left\{  u=0\right\}  $ is empty, then by applying~\cite[Theorem~1.1]{Farina03} or by constructing suitable
comparison functions as in the first part of the proof of Theorem~\ref{main2} (see Section~\ref{sec2} below), one knows that $u\equiv\pm1$ in $\R^n$. Therefore, if $u$ is non-constant, then $\left\{  u=0\right\}\neq\emptyset$ and $-1<u<1$ in $\R^n$ from the strong maximum principle. Moreover, if $u\ge0$ (resp. $u\le0$) in $\R^n$ and $\{u=0\}\neq\emptyset$, then $u\equiv 0$ in $\R^n$ from the strong maximum principle. Hence we can assume
without loss of generality that the three sets $\left\{  u=0\right\}$, $\{u>0\}$ and $\{u<0\}$ are not empty.

We point out that here it is not assumed that the nodal set $\{ u=0\}$ is a graph, that is, the sets $\{u>0\}$ or~$\{u<0\}$ are not assumed to be epigraphs. For rigidity results in the epigraph case we refer to~\cite{Farina0, Farina3} and the references therein.

As an application of Theorem~\ref{main}, using the classification result of
stable solutions of the Allen-Cahn equation in the plane, we get the following
\emph{strong half-space theorem}:

\begin{corollary}{\rm{(Strong half-space theorem)}}\label{strong}
Suppose $n=2$. Let $u_{1}<u_{2}$ be two non-constant solutions of~\eqref{1} in
$\mathbb{R}^{2}$. Then $u_{1}$ and $u_{2}$ are one-dimensional, namely there exist a unit vector $e$ and some real numbers $a<b$ such that $u_1(x)=H(x\cdot e+a)$ and $u_2(x)=H(x\cdot e+b)$ for all $x\in\R^2$.
\end{corollary}

We will also generalize Theorem~\ref{main} to a free boundary problem. We refer
to Section~\ref{free} for the precise statement and its proof.

\medskip

Our results are inspired by analogous results in the minimal surface theory.
A half-space theorem for minimal surfaces in $\mathbb{R}^{3}$ was proved by
Hoffman and Meeks~\cite{HM}. It states that connected, proper, minimal surfaces
in $\mathbb{R}^{3}$ are necessarily planes. A version of a half-space theorem for
minimal surfaces with bounded Gaussian curvature is proved in~\cite{X}. The
half-space theorem plays an important role in the understanding of the
structure of minimal spaces, and there is a vast literature on this subject. It is
used in the proof of the local removal singularity theorem~\cite{MPR}, it is
also used to study the properness of minimal surfaces (see, for instance,
\cite{MR}). {In~\cite{CM}, Colding and Minicozzi proved that the plane is the only
complete embedded minimal disk in $\mathbb{R}^{3}$}, by establishing a
chord-arc bound and applying Hoffman-Meeks half-space theorem.

We remark that the half-space theorem is not true for minimal hypersurfaces in
$\mathbb{R}^{n}$ with $n\geq4$. For example the higher dimensional catenoid
provides a counterexample. However, for the Allen-Cahn equation, this question is still
open in higher dimensions. In view of the construction of solutions
concentrated on higher dimensional catenoid~\cite{ADW}, we turn to believe
that the half-space theorem of the Allen-Cahn equation should be true also for all $n\geq
4$. Intuitively, for solutions of the Allen-Cahn equation there are strong
interations between different ends, while this is not the case for minimal surfaces.

The proof of the half-space theorem for minimal surfaces uses sweeping principle.
It appears that this idea does not work for the Allen-Cahn case, although we can
prove partial results along this direction. Our main result for solutions of the Allen-Cahn equation in arbitrary dimension with zero level set contained in a half-space is the following:

\begin{theorem}\label{main2}
Let $n\ge 1$ and $u$ be a non-constant solution of the Allen-Cahn
equation~\eqref{1} in~$\mathbb{R}^n$. If $u<0$ in the
half-space $\{x_n<0\}$, then there exists $a\in
\mathbb{R}$ such that
\[
u(x) \leq H(x_n+a)
\]
for all $x \in\mathbb{R}^{n}$, and either
\begin{enumerate}
\item $u(x) = H(x_n+a)$ for all
$x\in\mathbb{R}^n$, or
\item $u(x) < H(x_n+a)$ for all
$x\in\mathbb{R}^n$ and there exists a sequence $(y_{k})_{k\in\mathbb{N}}$ in
$\mathbb{R}^{n-1}\times\{0\}$ such that $|y_k|\to+\infty$ as $k\to+\infty$, and the functions $u(\cdot+y_k)$
converge in $C^2_{loc}(\mathbb{R}^{n})$ to the function $x\mapsto H(x_n+a)$ as $k\to+\infty$.
\end{enumerate}
\end{theorem}

In Theorem~\ref{main2} and throughout the paper, $|\ |$ denotes the Euclidean
norm in~$\mathbb{R}^{n}$. We also use the following notations:
$B_{R}(x)=\big\{y\in\mathbb{R}^{n}: |y-x|<R\big\}$ and
$B_{R}=B_{R}(0)$ for $R>0$ and $x\in\mathbb{R}^{n}$, and we denote
$${\rm{dist}}(x,E)=\inf_{y\in E}|x-y|$$
for $x\in\R^n$ and $E\subset\R^n$.

We complete the introduction by listing some corollaries following from Theorem~\ref{main2}.

\begin{corollary}\label{cor2}
Let $n\ge 1$ and $u$ be a non-constant solution of the Allen-Cahn
equation~\eqref{1} in~$\mathbb{R}^{n}$. Then there does not exist a non-degenerate
cone containing $\{u=0\}$.
\end{corollary}

\begin{corollary}\label{cor4}
Let $n\ge 1$ and $u$ be a non-constant solution of the Allen-Cahn
equation~\eqref{1} in $\mathbb{R}^{n}$. If there exists a closed half-space
$E$ such that $\{u=0\}\subset E$ and $\{u=0\}\cap\partial E
\neq\emptyset$, then there is $a\in\mathbb{R}$ such that $u(x) =H(x\cdot e+a)$
for all $x\in\mathbb{R}^{n}$, where $e$ is a unit vector orthogonal to $\partial E$.
\end{corollary}

It also follows from Theorem~\ref{main2} that, in any dimension $n\ge1$, if the zero level set of a non-constant solution of~\eqref{1} is bounded in a unit direction $e$, then $u$ is identically equal to $H(\pm x\cdot e+a)$ for some $a\in\R$. This last result (see Corollary~\ref{cor3} at the end of Section~\ref{sec2} below) corresponds to Theorem~1.1 obtained by Farina~\cite{Farina03}. We provide in Section~\ref{sec2} another proof using Theorem~\ref{main2}. We remark also that this result points out the difference between minimal surfaces and level sets of Allen-Cahn solutions from the point of view of half-space theorems: for $n\geq 4$ the minimal catenoid is contained in a slab, while entire solutions of the Allen-Cahn equation cannot have the zero level set contained in a slab, unless they are constant or equal to $H(x_n)$ up to translation and rotation of the variables. In particular, the zero level set of the solutions obtained in \cite{ADW} is not contained in a slab, although such level set approaches the minimal catenoid in a compact region.

It should be interesting to link such kind of half-space results to the De Giorgi conjecture. In this sense, an open question is to understand if Theorem~\ref{main} can be true in all dimensions, at least with the hypothesis of the monotonicity of the solution in one direction. Note, in particular, that the zero level sets of the $x_n$-monotone and non-planar solutions of~\eqref{1} in $\R^n$ with $n\ge9$, constructed by Del Pino, Kowalczyk and the fifth author in~\cite{M2}, are not included in any half-space.

\begin{remark}{\rm Theorem~$\ref{main}$ and Corollary~$\ref{cor2}$, as well as Corollary~$\ref{cor3}$ below, do not hold good if the zero level set $\{u=0\}$ is replaced by another level set $\{u=\mu\}$ with $\mu\neq0$. Similarly, Theorem~$\ref{main2}$ does not hold good either if one assumes that $u<\mu$ in $\{x_n<0\}$ for some $\mu>0$. Indeed, equation~\eqref{1} admits solutions $u_L$ which vanish on $(L\mathbb{Z})^n$ with $L>\pi\sqrt{n}$, are $2L$-periodic with respect to each variable $x_i$, which satisfy $\max_{\R^n}|u_L|\to 0$ as $L\displaystyle{\mathop{\to}^>}\pi\sqrt{n}$, and which are not one-dimensional!}
\end{remark}

\noindent{\bf{Outline of the paper.}} Section~\ref{sec2} is devoted to the proof of Theorem~\ref{main2} and its corollaries. The proof of Theorem~\ref{main2} is itself used in the proof of Theorem~\ref{main} and  Corollary \ref{strong} done in Section~\ref{sec3}. Lastly, Section~\ref{free} is concerned with a half-space theorem for a related free boundary problem.


\section{Half-space theorems in general dimension: proof of Theorem~\ref{main2} and its corollaries}\label{sec2}

The half-space theorem of minimal hypersurfaces in $\mathbb{R}^{n}$ is not true
when $n\geq4,$ because the higher dimensional catenoids lie in a half-space.
For the Allen-Cahn equation~\eqref{1}, we still do not know whether there is version of
the half-space theorem in dimension $n\geq4$. We have obtained partial
classification results in this direction, based on the maximum principle.

\medskip

We start this section with a general property holding in any dimension $n\ge1$. This can essentially be found in~\cite[Lemmas~3.2 and~3.3]{BCN2} and~\cite[Lemma~2.3]{Farina0}.

\begin{proposition}\label{lemdist}
Let $n\ge 1$ and $u$ be a non-constant solution of~\eqref{1} in any dimension $n\ge1$. Then $\{u=0\}\neq\emptyset$ and $|u(x)|\to1$ as ${\rm{dist}}(x,\{u=0\})\to+\infty$.
\end{proposition}

\begin{proof}
The proof is standard, we briefly sketch it for the sake of completeness. We recall from the introduction that $u$ necessarily satisfies $-1<u<1$ in $\R^n$ and $\{u=0\}\neq\emptyset$. Consider any $R>0$ and $x\in\R^n$ such that
$${\rm{dist}}(x,\{u=0\})>R$$
(hence, $\overline{B(x,R)}\cap\{u=0\}=\emptyset$).

Let $\lambda_R$ be the principal eigenvalue of $-\Delta$ in the ball $B_{R}$ with Dirichlet boundary condition, that is, there is a function $\varphi_R\in C^{2}(\overline{B_{R}})$ such that $-\Delta\varphi_R=\lambda_R\varphi_R$ in $\overline{B_{R}}$, $\varphi_R=0$ on~$\partial B_{R}$, and $\varphi_R>0$ in~$B_{R}$. From the classical radial symmetry result~\cite{GNN}, the function $\varphi_R$ is radially symmetric and decreasing with respect to the distance from the origin. Therefore, up to multiplication by a positive constant, one can assume without loss of generality that $\varphi_R(0)=1=\max_{\overline{B_R}}\varphi_R$. Notice also that $\lambda_R=\lambda_1/R^2$. Since the conclusion is concerned with the limit as ${\rm{dist}}(x,\{u=0\})\to+\infty$, one can assume without loss of generality that $R>0$ is large enough so that $0<\lambda_R<1$.

In the closed ball $\overline{B(x,R)}$, the continuous function $u$ does not vanish. Up to changing~$u$ into~$-u$, let us assume without loss of generality that $u>0$ in $\overline{B(x,R)}$. There exists then $\varepsilon_0$ such that $\varepsilon\varphi_R(\cdot-x)<u$ in $\overline{B(x,R)}$, for all $\varepsilon\in[0,\varepsilon_0]$. Furthermore, for every $\varepsilon\in[0,\sqrt{1-\lambda_R}]$, the function $\varepsilon\varphi_R(\cdot-x)$ satisfies
$$\Delta(\varepsilon\varphi_R(\cdot-x))+\varepsilon\varphi_R(\cdot-x)-(\varepsilon\varphi_R(\cdot-x))^3=\varepsilon\varphi_R(\cdot-x)\times(1-\lambda_R-(\varepsilon\varphi_R(\cdot-x))^2)\ge0$$
in $\overline{B(x,R)}$, since $0\le\varphi_R\le 1$ in $\overline{B_R}$. As a consequence, for any such $\varepsilon$, the function $\varepsilon\varphi_R(\cdot-x)$ is a subsolution of~\eqref{1} in the closed ball $\overline{B(x,R)}$ and it vanishes on $\partial B(x,R)$, while $u>0$ in $\overline{B(x,R)}$. It follows from the strong maximum principle that $\varepsilon\varphi_R(\cdot-x)<u$ in $\overline{B(x,R)}$ for every $\varepsilon\in[0,\sqrt{1-\lambda_R}]$. In particular, $u(x)>\sqrt{1-\lambda_R}$ since $\varphi_R(0)=1$. Since $\lambda_R=\lambda_1/R^2\to0$ as $R\to+\infty$ and $-1<u<1$ in $\R^n$, the conclusion follows.
\end{proof}

\begin{remark}{\rm Proposition~$\ref{lemdist}$ implies that $\sup_{\R^n}|u|=1$ if $\sup_{x\in\R^n}{\rm{dist}}(x,\{u=0\})=+\infty$. However, remember that the property $\sup_{x\in\R^n}{\rm{dist}}(x,\{u=0\})=+\infty$ is not always satisfied, since the Allen-Cahn equation admits non-trivial periodic solutions $u$ oscillating around the value $0$, and for which $\sup_{\R^n}|u|<1$.}
\end{remark}

We are now in position to prove our Theorem~$\ref{main2}$.

\begin{proof}[Proof of Theorem~$\ref{main2}$]  Throughout the proof, $u$ is a non-constant solution of~\eqref{1} such that
$$u<0\ \hbox{ in }\mathbb{R}^{n}_{-}=\{x_n<0\}.$$
As recalled in the introduction, we know that $-1<u<1$ in $\R^n$.

By Proposition~$\ref{lemdist}$, we have that
\begin{equation}
u(x_{1},\cdots,x_{n})\rightarrow-1\ \hbox{ as }x_{n}\rightarrow-\infty
\hbox{ uniformly in }(x_{1},\cdots,x_{n-1})\in\mathbb{R}^{n-1}.
\label{uniform}%
\end{equation}
since the function $u$ is negative in $\mathbb{R}^n_-$.

Denote $U(x)=H(x_n)=\tanh(x_{n}/\sqrt{2})$ and
\[
U^{\omega}(x)=U(x_{1},\cdots,x_{n-1},x_{n}+\omega)=H(x_n+\omega)=\tanh\left(  \frac
{x_{n}+\omega}{\sqrt{2}}\right)
\]
for $x\in\mathbb{R}^{n}$ and $\omega\in\mathbb{R}$. We shall now show that
$u\le U^{\omega}$ in $\mathbb{R}^{n}$ for all $\omega$ large enough. To do so,
let $A>0$ be such that
\begin{equation}
\label{uU}u\le-\frac{1}{\sqrt{3}}\ \hbox{ in }\mathbb{R}^{n-1}\!\times
\!(-\infty,-A]\ \ \hbox{ and }\ \ U\ge\frac{1}{\sqrt{3}}
\ \hbox{ in }\mathbb{R}^{n-1}\!\times\![A,+\infty).
\end{equation}
We claim that
\[
u\le U^{\omega}\ \hbox{ in }\mathbb{R}^{n}\ \hbox{ for all }\omega\ge2A.
\]
To do so, pick any $\omega\in[2A,+\infty)$. We shall prove that $u\le
U^{\omega}$ in $\mathbb{R}^{n-1}\times(-\infty,-A]$. Assume by way of
contradiction that $M:=\sup_{\mathbb{R}^{n-1}\times(-\infty,-A]}(u-U^{\omega
})>0$. Then there is a sequence $(z_{k})_{k\in\mathbb{N}}=(z^{\prime}_{k},z_{k,n})_{k\in\mathbb{N}}$
in $\mathbb{R}^{n-1}\times(-\infty,-A]$ such that
\[
u(z_{k})-U^{\omega}(z_{k})\to M>0\ \hbox{ as }k\to+\infty.
\]
By~\eqref{uniform}, the sequence $(z_{k,n})_{k\in\mathbb{N}}$ is then bounded.
Furthermore, by uniform continuity of~$U$ (or of $u$), property~\eqref{uU} and
the assumption~$\omega\ge2A$ imply that
\[
\limsup_{k\to+\infty}z_{k,n}<-A.
\]
Therefore, there is $\zeta\in(-\infty,-A)$ such that, up to extraction of a
subsequence, $z_{k,n}\to\zeta$ as $k\to+\infty$. Up to extraction of another
subsequence, the functions $x\mapsto u(x^{\prime}+z^{\prime}_{k},x_{n})$
converge in $C^{2}_{loc}(\mathbb{R}^{n})$ as $k\to+\infty$ to a solution
$w_{\infty}$ of~\eqref{1} such that $w_{\infty}\le-1/\sqrt{3}$ and
$w_{\infty}-U^{\omega}\le M$ in $\mathbb{R}^{n-1}\times(-\infty,-A]$, while
$w_{\infty}(0,\zeta)-U^{\omega}(0,\zeta)=M$. At the (interior) point
$(0,\zeta)\in\mathbb{R}^{n-1}\times(-\infty,-A)$, there holds
\begin{equation}
\label{winfty}0\ge\Delta(w_{\infty}-U^{\omega})(0,\zeta)=-w_{\infty}%
(0,\zeta)+w_{\infty}(0,\zeta)^{3}+U^{\omega}(0,\zeta)-(U^{\omega}%
(0,\zeta))^{3}.
\end{equation}
But $-1<U^{\omega}(0,\zeta)<U^{\omega}(0,\zeta)+M=w_{\infty}(0,\zeta
)\le-1/\sqrt{3}$ and the function $s\mapsto s-s^{3}$ is decreasing in
$[-1,-1/\sqrt{3}]$. Hence the right-hand side of~\eqref{winfty} is positive, a
contradiction. Therefore,
$$\sup_{\mathbb{R}^{n-1}\times(-\infty,-A]}(u-U^{\omega})\le0,$$
that is, $u\le U^{\omega}$ in $\mathbb{R}^{n-1}\times(-\infty,-A]$. Similarly,
since $U^{\omega}\ge1/\sqrt{3}$ in $\mathbb{R}^{n-1}\times[-A,+\infty)$ and
$U^{\omega}\ge u$ on $\mathbb{R}^{n-1}\times\{-A\}$, while the function
$s\mapsto s-s^{3}$ is decreasing in $[1/\sqrt{3},1]$, one can show that $u\le
U^{\omega}$ in $\mathbb{R}^{n-1}\times[-A,+\infty)$. Finally,
\[
u\le U^{\omega}\ \hbox{ in }\mathbb{R}^{n}\ \hbox{ for all }\omega\ge2A.
\]

Define now
\[
a=\inf\big\{\omega\in\mathbb{R}: u\le U^{\omega}\hbox{ in }\mathbb{R}%
^{n}\big\}.
\]
The previous paragraph yields $a\le2A$. On the other hand, since
$U^{\omega}\to-1$ as $\omega\to-\infty$ (at least) pointwise in $\mathbb{R}%
^{n}$, while $u>-1$ in $\mathbb{R}^{n}$, one infers that $a\in\mathbb{R}$.
By continuity, there holds $u\le U^{a}$ in $\mathbb{R}^{n}$, that is,
\begin{equation}
\label{inequtanh}u(x)\le H(x_n+a)
\ \hbox{ for all }x\in\mathbb{R}^{n}.
\end{equation}
This statement corresponds to the first part of the conclusion of
Theorem~\ref{main2}.

Let us now show the second part of the conclusion. First of all, if there is a
point $x^{*}\in\mathbb{R}^{n}$ such that $u(x^{*})=U^{a}(x^{*})=H(x^*_n+a)$, then the
strong maximum principle implies that $u\equiv U^{a}$ in $\mathbb{R}^{n}$,
that is,
\[
u(x)=H(x_n+a) \ \hbox{ for all }x\in
\mathbb{R}^{n}.
\]

Let us then assume in the sequel that $u<U^{a}$ in $\mathbb{R}^{n}$, that
is,
\begin{equation}\label{strict}
u(x)<U^a(x)=H(x_n+a)
\ \hbox{ for all }x\in\mathbb{R}^{n}.
\end{equation}
Let $B>0$ be such that $U^{a}\ge2/3\,(>1/\sqrt{3})$ in
$\mathbb{R}^{n-1}\times[B,+\infty)$. We claim that
\begin{equation}
\label{AB}\sup_{\mathbb{R}^{n-1}\times[-A,B]}\big(u-U^{a})=0.
\end{equation}
Indeed, otherwise, one would have $\sup_{\mathbb{R}^{n-1}\times[-A,B]}
(u-U^{a})<0$ and, by uniform continuity of $U$, there
would exist $\omega\in(-\infty,a)$ such that $u\le U^{\omega}$ in
$\mathbb{R}^{n-1}\times[-A,B]$ and $U^{\omega}\ge1/\sqrt{3}$ in
$\mathbb{R}^{n-1}\times[B,+\infty)$. With the same arguments as in the previous
paragraph, one gets that $u\le U^{\omega}$ in $\mathbb{R}^{n-1}\times
(-\infty,-A]$ and $u\le U^{\omega}$ in $\mathbb{R}^{n-1}\times[B,+\infty)$. As
a consequence, $u\le U^{\omega}$ in $\mathbb{R}^{n}$, contradicting the
minimality of $a$. Therefore,~\eqref{AB} holds.

From~\eqref{AB}, one infers the existence of a sequence
$(\xi_{k})_{k\in\mathbb{N}}=(\xi^{\prime}_{k},\xi_{k,n})_{k\in\mathbb{N}}$ in
$\mathbb{R}^{n-1}\times[-A,B]$ such that
\begin{equation}
\label{uxin}u(\xi_{k})-U^{a}(\xi_{k})\to0\ \hbox{ as }k\to+\infty.
\end{equation}
Up to extraction of a subsequence, one can assume that $\xi_{k,n}\to\xi_{\infty,n}$
as $k\to+\infty$, for some~$\xi_{\infty,n}\in[-A,B]$. Notice that
$|\xi^{\prime}_{k}|\to+\infty$ as $k\to+\infty$,
since otherwise there would exist a point $\xi\in\mathbb{R}^{n-1}\times[-A,B]$
such that $u(\xi)=U^{a}(\xi)$, contradicting~\eqref{strict}. Denote
\[
y_{k}=(\xi^{\prime}_{k},0)\in\R^{n-1}\times\{0\}\ \hbox{ and }
\ u_{k}(x)=u(x+y_{k})\ \hbox{ for }k\in\mathbb{N}\hbox{ and }x\in\mathbb{R}^{n}.
\]
To complete the proof of Theorem~\ref{main2}, we just need to show that
\[
u_{k}(x)\to H(x_n+a)
\]
in $C^2_{loc}(\mathbb{R}^{n})$ as $k\to+\infty$. Up to extraction
of a subsequence, the functions $u_{k}$ converge in~$C^{2}_{loc}(\mathbb{R}^{n})$
to a solution $u_{\infty}:\mathbb{R}^{n}\to[-1,1]$
of~\eqref{1} such that $u_{\infty}(x)\le H(x_n+a)$ in
$\mathbb{R}^{n}$ from~\eqref{inequtanh} and the definition of $y_{k}$.
Furthermore, $u_{\infty}(0,\xi_{\infty,n})=U^{a}(0,\xi_{\infty,n})=H(\xi_{\infty,n}+a)$ by~\eqref{uxin}.
It then follows from the strong maximum
principle that $u_{\infty}(x)=H(x_n+a)$ for all~$x\in
\mathbb{R}^{n}$. Furthermore, the limit of the functions $u_{k}$ being
independent of the subsequence, one concludes that the whole sequence
$(u_{k})_{k\in\mathbb{N}}$ converges to the function $x\mapsto H(x_n+a)$
in~$C^{2}_{loc}(\mathbb{R}^{n})$ as $k\to+\infty$. The proof of
Theorem~\ref{main2} is thereby complete.
\end{proof}

\begin{remark}\label{remexp}{\rm In Theorem~$\ref{main2}$, one has $-1<u(x)<H(x_n+a)=\tanh((x_n+a)/\sqrt{2})$ for all~$x\in\R^n$. Modica's inequality $|\nabla u|^2\le(1-u^2)^2/2$ (see~\cite{Modica}) then yields
\begin{equation}\label{ineqexp}
|\nabla u(x)|\le 2\,\sqrt{2}\,e^{\sqrt{2}\,(x_n+a)}\hbox{ for all }x\in\R^n.
\end{equation}}
\end{remark}

\medskip

By changing $u$ into $-u$ in Theorem~\ref{main2}, the following result immediately follows.

\begin{theorem}\label{main3}
Let $n\ge 1$ and $u$ be a non-constant solution of the Allen-Cahn
equation~\eqref{1} in~$\mathbb{R}^n$. If $u>0$ in the
half-space $\{x_n<0\}$, then there exists $a\in
\mathbb{R}$ such that
\[
u(x) \geq H(-x_n+a)
\]
for all $x \in\mathbb{R}^{n}$, and either $u(x) = H(-x_n+a)$ for all $x\in\mathbb{R}^n$, or
$u(x) > H(-x_n+a)$ for all~$x\in\mathbb{R}^n$ and there exists a sequence $(y_{k})_{k\in\mathbb{N}}$ in
$\mathbb{R}^{n-1}\times\{0\}$ such that $|y_k|\to+\infty$ as~$k\to+\infty$, and the
functions $u(\cdot+y_k)$ converge in $C^2_{loc}(\mathbb{R}^{n})$ to the function
$x\mapsto H(-x_n+a)$ as~$k\to+\infty$.
\end{theorem}

Let us now turn to the proof of Corollaries~\ref{cor2} and~\ref{cor4}, which follow from Theorems~\ref{main2} and~\ref{main3}.

\begin{proof}[Proof of Corollary~$\ref{cor2}$]
Assume by way of contradiction that $u$ is a non-constant solution of~\eqref{1} with~$\{u=0\}$ contained in a non-degenerate cone. Then, up to changing $u$ into $-u$, and up to translation and rotation of the variables, it follows that
\begin{equation}\label{hypcone}
u<0\ \hbox{ in }\big\{(x',x_n)\in\R^{n-1}\times\R:x_n<\beta\,|x'|\big\},
\end{equation}
for some $\beta>0$. Theorem~\ref{main2}, together with~\eqref{hypcone}, then yields the existence of a real number~$a$ such that $u(x)=H(x_n+a)$ for all $x\in\R^n$, which leads to a contradiction.
\end{proof}

\begin{proof}[Proof of Corollary~$\ref{cor4}$]
Up to changing $u$ into $-u$ and up to translation and rotation of the variables, one can also assume without loss of generality that $e=(0,\cdots,0,1)$, that~$u<0$ in $\{x_n<0\}$ and that $u(x',0)=0$ for some $x'\in\R^{n-1}$. Theorem~\ref{main2} then implies that $u(x)\le H(x_n+a)$ in $\R^n$, for some $a\in\R$, and the other parts of the conclusion hold for that real number~$a$. We claim that $u(x',0)=H(a)$. Indeed, if not, then $0=u(x',0)<H(a)$, hence $a>0$, while Theorem~\ref{main2} also yields the existence of a sequence $(y_k)_{k\in\N}=(y'_k,0)_{k\in\N}$ in $\R^{n-1}\times\{0\}$ such that $u(\cdot+y_k)\to H(x_n+a)$ in $C^2_{loc}(\R^n)$ as $k\to+\infty$. In particular, $u(y'_k,-a/2)\to H(a/2)>0$ as $k\to+\infty$, hence $u(y'_k,-a/2)>0$ for all $k$ large enough, which is impossible since $u<0$ in $\{x_n<0\}$. Therefore, $u(x',0)=H(a)$ (hence, $a=0$) and $u(x)\equiv H(x_n)$ in $\R^n$ from Theorem~\ref{main2}.
\end{proof}

Lastly, as announced at the end of Section~\ref{sec1}, we can retrieve from Theorem~\ref{main2} the one-dimensional property of any non-constant solution of~\eqref{1} whole zero level set is contained in a slab. This result was obtained by Farina~\cite[Theorem~1.1]{Farina03}.

\begin{corollary}\label{cor3}{\rm{\cite[Theorem~1.1]{Farina03}}}
Let $n\ge 1$ and $u$ be a non-constant solution of the Allen-Cahn equation~\eqref{1} in~$\mathbb{R}^{n}$. If $\{u=0\}$ is contained in a slab $\{x\in\R^n:|x\cdot e|<A\}$ for some unit vector~$e$ and some real number $A>0$, then there exists $a\in\mathbb{R}$ such that either $u(x)=H(-x\cdot e+a)$ or $u(x)=H(x\cdot e+a)$, for all $x\in\mathbb{R}^{n}$.
\end{corollary}

We here give a proof using Theorem~\ref{main2}. 

\begin{proof}
Up to changing $u$ into $-u$ and up to translation and rotation of the variables, one can also assume without loss of generality that $e=(0,\cdots,0,1)$, that $u<0$ in $\{x_n\le 0\}$ and that
$$\{u=0\}\subset\{0<x_n<2A\}.$$
It then follows from Theorem~\ref{main2} that there exists $b\in\R$ such that
$$u(x)\le H(x_n+b)$$
for all $x\in\R^n$ and either $u(x)=H(x_n+b)$ for all $x\in\R^n$, or $u(\cdot+y_k)\to H(x_n+b)$ in~$C^2_{loc}(\R^n)$ as $k\to+\infty$, for some sequence $(y_k)_{k\in\N}$ in $\R^{n-1}\times\{0\}$. In both cases, one has~$\sup_{\R^{n-1}\times(-\infty,2A]}u\le H(2A+b)<1$ and $\sup_{\R^n}u=\sup_{\R}H=1$. By continuity, one infers that $u>0$ in $\{x_n\ge 2A\}$. From Theorem~\ref{main2} applied to the solution
$$x=(x',x_n)\mapsto-u(x',-x_n+2A),$$
there exists then $c\in\R$ such that $-u(x',-x_n+2A)\le H(x_n+c)$ for all $x\in\R^n$, hence
$$u(x)\ge H(x_n-2A-c)$$
for all $x\in\R^n$. Finally $H(x_n-2A-c)\le u(x)\le H(x_n+b)$ for all $x\in\R^n$ and one concludes from~\cite[Theorem~3.1]{BH} or~\cite[Theorem~2.1]{Farina0} that $u(x)\equiv H(x_n+a)$ in $\R^n$, for some $a\in\R$.
\end{proof}


\section{Proof of the half-space theorem in dimensions $n=2,3$}\label{sec3}

As we mentioned in Section~\ref{sec1}, the proof of the half-space theorem for minimal
surfaces uses the family of catenoids and the sweeping principle. In the
Allen-Cahn case, the solutions are defined in the whole space, and it is not
easy to apply this idea.

We remark that  the $n=2$
case of Theorem~\ref{main} can also be proven by applying the method in De
Silva and Savin~\cite{DeS-S}. Our proof uses Pohozaev identity (also called
balancing condition, see~\cite{DKP}) and is very
different from theirs.

We shall prove Theorem~\ref{main} for $n=1,2,3$. Notice that if $n=1,$ then the solution $u$ is trivially one-dimensional and since it is not constant, it is then equal to $H(x_1)$ up to shifts, as follows directly from ODE analysis. The cases of $n=2,3$ are more complicated. Although we can deal with these two cases in a unified
way, we choose to first give a simple proof when $n=2$, because this
gives us a clear geometric intuition behind the whole proof.


\subsection{The case $n=2$}
Let $u$ be a solution whose zero level set $\{u=0\}$ is contained in the
half-space $\{x_2>0\}$. Up to changing $u$ into $-u$ and/or $x_1$ into $-x_1$, and shifting in the direction~$x_2$, we may assume without loss of generality that $u<0$ in $\{x_{2}<0\}$ and, from Theorem~\ref{main2}, that
\begin{equation}\label{uH}
u(x_1,x_2)\le H(x_2)\ \hbox{ for all }(x_1,x_2)\in\R^2
\end{equation}
and that there exists a sequence $(t_{k}^{+})_{k\in\N}$ such that
$t_{k}^{+}\rightarrow+\infty$ and $u(x_1+t^+_k,x_2)\to H(x_2)$ in $C^2_{loc}(\R^2)$ as $k\to+\infty$.

For each $x_{1}\in\R,$ we define
\[
g(x_{1})  =\inf\left\{  x_2\in\R:u(x_{1},x_2)  =0\right\}\ \in[0,+\infty].
\]
Note that the infimum is a minimum if $g(x_1)$ is a real number. Note also that $g(x_1)$ might a priori be $+\infty$ for some values $x_1$, in which case $u(x_1,x_2)<0$ for all $x_2\in\R$ (nevertheless, the conclusion $u(x_1,x_2)\equiv H(x_2)$ in $\R^2$ will show that this case is impossible). We know at this point that $g$ cannot be equal to $+\infty$ on $(-\infty,\xi)$ for some $\xi\in\R$ since otherwise the zero level set of $u$ would be included in the quarter-plane $\{x_1\ge\xi,\, x_2\ge0\}$, which is ruled out by Corollary~\ref{cor2}. Let us set
\[
\alpha=\underset{x_{1}\rightarrow-\infty}{{\lim\inf}}\, g(x_{1})\ \in[0,+\infty].
\]

Let us first consider the case $0\le\alpha<+\infty$. There exists then a sequence $(t_{k}^{-},s_{k}^{-})_{k\in\N}$ such that $u(t_{k}^{-},s_{k}^{-})=0$, and
$t_{k}^{-}\rightarrow-\infty$ and $s_{k}^{-}\rightarrow\alpha$ as $k\to+\infty$. Up to extraction of a subsequence, the functions $(x_1,x_2)\mapsto u(x_1+t^-_k,x_2+\alpha)$ converge in $C^2_{loc}(\R^2)$ to a classical solution $u_\infty$ of~\eqref{1} such that $u_\infty(0,0)=0$ and $u_\infty\le0$ in $\{x_2\le0\}$, owing to the definition of $\alpha$. Furthermore, $u_\infty(x_1,x_2)\le H(x_2+\alpha)$ in $\R^2$ from~\eqref{uH}, hence $u_\infty(x_1,-\infty)=-1$ for each $x_1\in\R$, and $u_\infty<0$ in $\{x_2<0\}$ from the strong maximum principle. Corollary~\ref{cor4} then implies that $u_\infty(x_1,x_2)\equiv H(x_2+b)$ in $\R^2$ for some real number $b$. Since $u_\infty(0,0)=0$, one finally infers that $b=0$ and $u_\infty(x_1,x_2)\equiv H(x_2)$ in $\R^2$.

Consider now the semi-infinite vertical strip
\[
\Omega_{k}:=\left\{(x_{1},x_{2})\in\R^2  :t_{k}^{-}<x_{1}<t_{k}^{+},-\infty<x_{2}<\alpha\right\}.
\]
Let $X=\left(  0,1\right)$. The balancing condition (see~\cite[Appendix]{DKP})
tells us that
\begin{equation}
\int_{\partial\Omega_{k}}\left[  \left(  \frac{1}{2}\left\vert \nabla
u\right\vert ^{2}+F\left(  u\right)  \right)  X\cdot\nu-\left(  \nabla u\cdot
X\right)  \left(  \nabla u\cdot\nu\right)  \right]  d\sigma=0, \label{b}%
\end{equation}
where
$$F(s)=\frac{(1-s^2)^2}{4}$$
and $\nu$ is the outward unit normal of the domain $\Omega_k$ (which is defined everywhere except at the corners $(t^\pm_k,\alpha)$. The previous formula means that
$$\begin{array}{l}
\displaystyle\int_{t^-_k}^{t^+_k}\left(\frac{|\nabla u(x_1,\alpha)|^2}{2}+F(u(x_1,\alpha))-u_{x_2}^2(x_1,\alpha)\right)\,dx_1\vspace{3pt}\\
\qquad\qquad\displaystyle+\underbrace{\int_{-\infty}^\alpha\big(u_{x_1}(t_k^-,x_2)\,u_{x_2}(t^-_k,x_2)-u_{x_1}(t_k^+,x_2)\,u_{x_2}(t^+_k,x_2)\big)\,dx_2}_{=:I_k}=0,\end{array}$$
where the second integral $I_k$ converges absolutely from Remark~\ref{remexp}. Since $u(x_1+t^-_k,x_2)\to H(x_2)$ and $u(x_1+t^-_k,x_2+\alpha)\to H(x_2)$ as $k\to+\infty$ in $C^2_{loc}(\R^2)$, together with Remark~\ref{remexp}, it follows that~$I_k\to0$ as $k\to+\infty$. Therefore,
$$\int_{t^-_k}^{t^+_k}\left(\frac{|\nabla u(x_1,\alpha)|^2}{2}+F(u(x_1,\alpha))-u_{x_2}^2(x_1,\alpha)\right)dx_1\to 0\ \hbox{ as }k\to+\infty.$$
On the other hand, Modica's
inequality
\[
F(u)\ge \frac{|\nabla u|^2}{2}\,\left(\ge \frac{u_{x_2}^2}{2}\right)
\]
(see~\cite{Modica}) implies that
$$\frac{|\nabla u|^2}{2}+F(u)-u_{x_2}^2=\frac{u_{x_1}^2}{2}+F(u)-\frac{u_{x_2}^2}{2}\ge\frac{u_{x_1}^2}{2}\ge0$$
in $\R^2$. Since $t^\pm_k\to\pm\infty$ as $k\to+\infty$, we finally get that
$$\frac{|\nabla u(x_1,\alpha)|^2}{2}+F(u(x_1,\alpha))-u_{x_2}^2(x_1,\alpha)=\frac{u_{x_1}^2(x_1,\alpha)}{2}=0$$
for all $x_1\in\R$, hence
\[
F(u(x_1,\alpha))=\frac{u_{x_2}^2(x_1,\alpha)}{2}=\frac{|\nabla u(x_1,\alpha)|^2}{2}
\]
for all $x_1\in\R$. One concludes from~\cite{CGS,Modica} that $u$ is one-dimensional, namely $u(x_1,x_2)\equiv H(x_2)$ in $\R^2$ from the first paragraph of the proof.

Let us finally consider the case $\alpha=+\infty$. Here, remembering also that $u<0$ in $\{x_2<0\}$, it follows from Proposition~\ref{lemdist} that
\[
\sup_{x_1\le -R,\,x_2\le0}u(x_1,x_2)\to-1
\]
as $R\to+\infty$, hence $|\nabla u(x_1,x_2)|\to0$ as $x_1\to-\infty$ uniformly with respect to $x_2\le0$, from standard elliptic estimates. Together with Remark~\ref{remexp}, this implies that
$$\int_{-\infty}^0u_{x_1}(x_1,x_2)\,u_{x_2}(x_1,x_2)\,dx_2\to0\ \hbox{ as }x_1\to-\infty.$$
Therefore, by applying~\eqref{b} in the region
\[
\left\{  \left(  x_{1},x_{2}\right)\in\R^2  :-k<x_{1}<t_{k}^{+},-\infty<x_{2}<0\right\}\, ,
\]
one gets with the same arguments as before that
$$\int_{-k}^{t^+_k}\left(\frac{|\nabla u(x_1,0)|^2}{2}+F(u(x_1,0))-u_{x_2}^2(x_1,0)\right)dx_1\to 0\ \hbox{ as }k\to+\infty.$$
This leads to the same conclusion as in the previous paragraph and the proof of Theorem~\ref{main} in the case $n=2$ is thereby complete.

\begin{proof}[Proof of Corollary \ref{strong}]
Let $u_1<u_2$ be two non-constant solutions of~\eqref{1} in
$\mathbb{R}^{2}$. Remember that $-1<u_1<u_2<1$ in $\R^2$. Using $u_1$ and $u_2$ as barriers and applying minimizing arguments, we can construct a stable solution $u_3$ of~\eqref{1} with
$$-1<u_1\leq u_3 \leq u_2<1\ \hbox{ in $\R^2$}.$$
In dimension two, stable solutions are one-dimensional, as follows from~\cite[Theorem~Ê1.8]{BCN}. The function $u_3$ is then one-dimensional stable and it takes values in $(-1,1)$, hence there exist then a unit vector $e$ and a real number $c$ such that $u_3(x)\equiv H(x\cdot e+c)$ in~$\R^2$. Therefore, the nodal set of $u_1$ is contained in the half-space $\{x\cdot e+c\ge0\}$ and $u_1<0$ in $\{x\cdot e+c<0\}$. By the weak half-space Theorem~\ref{main},~$u_1$ has to be one-dimensional, and more precisely there is $a\in\R$ such that $u_1(x)\equiv H(x\cdot e+a)$ in~$\R^2$. The same is true for~$u_2$, with $u_2(x)\equiv H(x\cdot e+b)$ in $\R^2$, for some real number $b$ such that $b>a$ (since $u_2>u_1$).
\end{proof}


\subsection{The case $n=3$}\label{n=3}

Next, we shall consider the case of dimension 3. The arguments in this section
can also be applied in the two dimensional case, but we preferred to use the more direct proof of the previous section in the case $n=2$.

First of all, as in the case $n=2$, up to changing $u$ into $-u$ and/or shifting in the direction $x_3$, we may assume without loss of generality that $u<0$ in $\{x_{3}<0\}$ and, from Theorem~\ref{main2}, that
\begin{equation}\label{uH2}
-1<u(x_1,x_2,x_3)\le H(x_3)\ \hbox{ for all }(x_1,x_2,x_3)\in\R^3
\end{equation}
and there is a sequence $(y_k)_{k\in\N}$ in $\R^2\times\{0\}$ such that $u(\cdot+y_k)\to H(x_3)$ in $C^2_{loc}(\R^3)$ as~$k\to+\infty$.

Now, let $A>0$ be such that
$$\tanh\left(\!-\frac{A}{\sqrt{2}}\right)\le-\sqrt{\frac{2}{3}}.$$
For $s>0$, let $\Omega_{s}$ be
the half-cylinder
\[
\Omega_s=\big\{ (x_1,x_2,x_3)\in\R^3  :x_1^{2}+x_2^{2}<s^2,\ x_3<-A\big\}
.
\]
Let $X=\left(  0,0,1\right)$. Then, for every $s>0$, the following balancing
formula (see~\cite{DKP}) holds in $\Omega_s$:
\begin{equation}\label{b2}
\int_{\partial\Omega_s}\left[  \left(  \frac{1}{2}\left\vert \nabla u\right\vert
^{2}+F\left(  u\right)  \right)  X\cdot\nu-\left(  \nabla u\cdot X\right)
\left(  \nabla u\cdot\nu\right)  \right]  d\sigma=0,
\end{equation}
meaning that
\begin{equation}\label{integrals}\begin{array}{rcl}
g(s) & \!\!\!\!:=\!\!\!\! & \displaystyle\!\!\int_{\{x_1^2+x_2^2<s^2\}}\!\!\!\left(\!\frac{|\nabla u(x_1,x_2,-A)|^2}{2}\!+\!F(u(x_1,x_2,-A))\!-\!u_{x_3}^2(x_1,x_2,-A)\!\right)dx_1dx_2\vspace{3pt}\\
& \!\!\!\!=\!\!\!\! & \displaystyle s\int_{-\infty}^{-A}\!\int_0^{2\pi}\!\big(u_{x_1}(s\cos\theta,s\sin\theta,x_3)\cos\theta+u_{x_2}(s\cos\theta,s\sin\theta,x_3)\sin\theta\big)\qquad\vspace{3pt}\\
& \!\!\!\!\!\!\!\! & \qquad\qquad\qquad\times\ u_{x_3}(s\cos\theta,s\sin\theta,x_3)\,d\theta\,dx_3.\end{array}
\end{equation}
Notice that the integrals converge absolutely from Remark~\ref{remexp}. As in the previous section, we infer from Modica's
inequality
\[
F(u)\ge\frac{|\nabla u|^2}{2}
\]
that
\begin{equation}\label{I}
\frac{|\nabla u|^2}{2}+F(u)-u_{x_3}^2=\frac{u_{x_1}^2+u_{x_2}^2}{2}+F(u)-\frac{u_{x_3}^2}{2}\ge\frac{u_{x_1}^2+u_{x_2}^2}{2}\ge0
\end{equation}
in $\R^3$. From this, we know that the function $g$ is nonnegative and nondecreasing in $(0,+\infty)$. Hence we can define
\begin{equation}\label{defalpha}
\alpha=\lim_{s\rightarrow+\infty}g(s)\ \in[0,+\infty] .
\end{equation}
Let us also define a function $K:(0,+\infty)\times\R\to\R$ by
\begin{equation}\label{defK}
K(s,x_3)  =\int_{\{x_1^2+x_2^2<s^2\}}\big(u_{x_1}^2(x_1,x_2,x_3)+u_{x_2}^2(x_1,x_2,x_3)\big)\,dx_1\,dx_2
\end{equation}
and notice from the definition of $g$ in~\eqref{integrals} and from~\eqref{I} that
\begin{equation}\label{gK}
K(s,-A)\le2\,g(s)\ \hbox{ for all }s>0.
\end{equation}

Remark~\ref{remexp} yields suitable
exponential decay of $u_{x_1}^2+u_{x_2}^2$ as $x_3\to-\infty$, from which we get the
following key-property of $K$.

\begin{lemma}
\label{K}There exist some constants $C>0$ and $\beta>0$, such that
\[
K(s,x_3)  \leq C\big(1+K(s+\beta\ln s,-A)\big)
,\text{ for all }s\geq1\hbox{ and }x_3\le-A.
\]
\end{lemma}

\begin{proof}
The functions $u_{x_1}$ and $u_{x_2}$ satisfy
\[
-\Delta u_{x_i}+(3u^{2}-1)\,u_{x_i}=0
\]
in $\R^3$. Thanks to~\eqref{uH2}, the function $u^{2}$ converges to $1$ as $x_3\to-\infty$, hence the operator $-\Delta+(3u^2-1)$ tends to $-\Delta+2.$ As a matter of fact, in $\R^2\times(-\infty,-A]$, one has
\[
-1<u\le H(-A)=\tanh \left(- \frac{A}{\sqrt{2}}\right)\le-\sqrt{\frac23}\,,
\]
thus $3u^2-1\ge1$ in $\R^2\times(-\infty,-A]$.

We are then going to compare $u_{x_1}$ (and later $u_{x_2}$) in $\R^2\times(-\infty,-A]$ to the bounded solution $\phi$ of the model problem
\[
\left\{
\begin{array}[c]{rcll}%
-\Delta\phi+\phi & = & 0, & (x_1,x_2,x_3)\in\R^2\times(-\infty,-A),\vspace{3pt}\\
\phi(x_1,x_2,-A)  & = & \left\vert u_{x_1}(x_1,x_2,-A)\right\vert, & (x_1,x_2)\in\R^2.
\end{array}
\right.
\]
By using the Fourier transform in the $(x_1,x_2)$ variables, the bounded solution $\phi$ of the above problem is given in $\R^2\times(-\infty,-A)$ by
\[
\phi(x_1,x_2,x_3)  =\int_{\mathbb{R}^{2}}\phi(x'_1,x'_2,-A)\,
G(x_1-x'_1,x_2-x'_2,x_3+A)\,dx'_1\,dx'_2,
\]
where, for $X_3<0$, $G(\cdot,\cdot,X_3)$ is the inverse Fourier transform of the function
\[
(\xi_1,\xi_2)\mapsto e^{\sqrt{\xi_1^2+\xi_2^2+1}\,X_3}\,,
\]
that is,
$$G(X_1,X_2,X_3)=\frac{1}{4\pi^2}\int_{\R^2}e^{i(X_1\xi_1+X_2\xi_2)+\sqrt{\xi_1^2+\xi_2^2+1}\,X_3}d\xi_1\,d\xi_2$$
for $(X_1,X_2,X_3)\in\R^2\times(-\infty,0)$. Note that we are interested in estimating $\phi(\cdot,\cdot,x_3)$ in the disks
\[
D_{s}:=\left\{ (x_1,x_2)\in\R^2  :x_1^{2}+x_2^{2}<s^2\right\},
\]
with $s\ge1$ and $x_3<-A$. To do so, for $s\ge1$ and $(x_1,x_2,x_3)\in\R^2\times(-\infty,-A)$, we divide~$\phi$ into two parts:
$$\begin{array}{rcl}
\phi(x_1,x_2,x_3) & = & \displaystyle\underbrace{\int_{\{|(x'_1,x'_2)|\ge s+\beta\ln s\}}\phi(x'_1,x'_2,-A)\,
G(x_1-x'_1,x_2-x'_2,x_3+A)\,dx'_1\,dx'_2}_{=:\phi_{1,s}(x_1,x_2,x_3)}\vspace{3pt}\\
& & \displaystyle+\underbrace{\int_{\{|(x'_1,x'_2)|<s+\beta\ln s\}}\phi(x'_1,x'_2,-A)\,
G(x_1-x'_1,x_2-x'_2,x_3+A)\,dx'_1\,dx'_2}_{=:\phi_{2,s}(x_1,x_2,x_3)},\end{array}$$
where $\beta>0$ will be chosen later.

On the one hand, since
\[
|\nabla u|\le \sqrt{2F(u)}\le\frac{1}{\sqrt{2}}
\]
in $\R^3$ from Modica's inequality~\cite{Modica}, and since $G\ge0$ in $\R^2\times(-\infty,0)$ and~$\phi\ge0$ in $\R^2\times(-\infty,-A)$ from the maximum principle, it follows that, for all $s\ge1$, $(x_1,x_2)\in D_s$ and $x_3<-A$,
$$\begin{array}{rcl}
0\le\phi_{1,s}(x_1,x_2,x_3) & \!\!\le\!\! & \displaystyle \frac{1}{\sqrt{2}}\int_{\{|(x'_1,x'_2)|\ge s+\beta\ln s\}}G(x_1-x'_1,x_2-x'_2,x_3+A)\,dx'_1\,dx'_2\vspace{3pt}\\
& \!\!\le\!\! & \displaystyle\underbrace{\frac{1}{\sqrt{2}}\int_{\{|(x'_1,x'_2)|\ge\beta\ln s\}}G(-x'_1,-x'_2,x_3+A)\,dx'_1\,dx'_2}_{=:\overline{\phi}_{s}(0,0,x_3)},\end{array}$$
where $\overline{\phi}_s$ denotes the bounded solution of
\begin{equation}\label{overphi}\left\{\begin{array}[c]{rcll}
-\Delta\overline{\phi}_s+\overline{\phi}_s & = & 0,\ \ (x_1,x_2,x_3)\in\R^2\times(-\infty,-A),\vspace{3pt}\\
\overline{\phi}_s(x_1,x_2,-A)  & = & \left\{\begin{array}{ll}\displaystyle\frac{1}{\sqrt{2}} & \hbox{if }|(x_1,x_2)|\ge\beta\ln s,\vspace{3pt}\\
0 & \hbox{otherwise}.\end{array}\right.\end{array}\right.
\end{equation}
It is immediate to check that there is a real number $\gamma>0$ small enough so that the function
\[
(x_1,x_2,x_3)\mapsto e^{\gamma\sqrt{x_1^2+x_2^2+1}}
\]
satisfies
\[
-\Delta\overline{\phi}+\overline{\phi}\ge0
\]
in $\R^3$. Hence, the function
\[
(x_1,x_2,x_3)\mapsto \frac{1}{\sqrt{2}}\,e^{\gamma\sqrt{x_1^2+x_2^2+1}-\gamma\sqrt{(\beta\ln s)^2+1}}
\]
is a supersolution of~\eqref{overphi}, and
$$0\le\overline{\phi}_s(x_1,x_2,x_3)\le\frac{e^{\gamma\sqrt{x_1^2+x_2^2+1}-\gamma\sqrt{(\beta\ln s)^2+1}}}{\sqrt{2}}\ \hbox{ for all $(x_1,x_2,x_3)\in\R^2\times(-\infty,-A)$}$$
from the maximum principle. As a consequence,
$$0\le\phi_{1,s}(x_1,x_2,x_3)\le\overline{\phi}_s(0,0,x_3)\le\frac{e^{\gamma-\gamma\beta\ln s}}{\sqrt{2}}$$
for all $s\ge1$, $(x_1,x_2)\in D_s$ and $x_3<-A$. Therefore, by choosing
$$\beta=\frac{2}{\gamma}>0,$$
one gets that
\begin{equation}\label{ineqphi1}
\int_{D_s}\phi_{1,s}^2(x_1,x_2,x_3)\,dx_1\,dx_2\le\frac{\pi\,e^{2\gamma}}{2\,s^2}\le\frac{\pi\,e^{2\gamma}}{2}\ \hbox{ for all $s\ge1$ and $x_3<-A$}.
\end{equation}

On the other hand, remember that $G\ge0$ in $\R^2\times(-\infty,0)$, and notice that $\|G(\cdot,\cdot,X_3)\|_{L^1(\R^2)}=e^{X_3}\le1$ for all $X_3<0$. Hence, using the Cauchy-Schwarz inequality and Fubini's theorem, we have, for all $s\ge1$ and $x_3<-A$,
$$\begin{array}{l}
\displaystyle\int_{D_s}\phi_{2,s}^{2}(x_1,x_2,x_3)\,dx_1\,dx_2\vspace{3pt}\\
\displaystyle\ =\int_{D_s}\left( \int_{\{|(x'_1,x'_2)|<s+\beta\ln s\}}\phi(x'_1,x'_2,-A)\,
G(x_1-x'_1,x_2-x'_2,x_3+A)\,dx'_1\,dx'_2\right)^{2}\,dx_1\,dx_2\\
\displaystyle\ \leq\int_{D_s}\left\{  \int_{\{|(x'_1,x'_2)|<s+\beta\ln s\}}\phi^{2}(x'_1,x'_2,-A)\,G(x_1-x'_1,x_2-x'_2,x_3+A)\,dx'_1\,dx'_2\right.\\
\displaystyle\qquad\qquad\qquad\qquad\quad\ \left.\times\int_{\{|(x'_1,x'_2)|<s+\beta\ln s\}}G(x_1-x'_1,x_2-x'_2,x_3+A)\,dx'_1\,dx'_2\right\}  dx_1\,dx_2\\
\displaystyle\ \leq\int_{\{|(x'_1,x'_2)|<s+\beta\ln s\}}\phi^{2}(x'_1,x'_2,-A)\,dx'_1\,dx'_2=\int_{D_{s+\beta\ln s}}\phi^{2}(x_1,x_2,-A)\,dx_1\,dx_2.
\end{array}$$

Therefore, together with~\eqref{ineqphi1}, it follows that, for all $s\ge1$ and $x_3<-A$,
$$\int_{D_s}\phi^2(x_1,x_2,x_3)\,dx_1\,dx_2\le\pi\,e^{2\gamma}+2\int_{D_{s+\beta\ln s}}\phi^{2}(x_1,x_2,-A)\,dx_1\,dx_2.$$
Lastly, since $3u^2-1\ge1$ in $\R^2\times(-\infty,-A]$, the maximum principle implies that $|u_{x_1}|\le\phi$ in $\R^2\times(-\infty,-A]$, hence
$$\begin{array}{rcl}
\displaystyle\int_{D_s}u_{x_1}^2(x_1,x_2,x_3)\,dx_1\,dx_2 & \le & \displaystyle\pi\,e^{2\gamma}+2\int_{D_{s+\beta\ln s}}\phi^{2}(x_1,x_2,-A)\,dx_1\,dx_2\vspace{3pt}\\
& = & \displaystyle\pi\,e^{2\gamma}+2\int_{D_{s+\beta\ln s}}u_{x_1}^{2}(x_1,x_2,-A)\,dx_1\,dx_2\end{array}$$
for all $s\ge1$ and $x_3<-A$. The same property holds similarly for the function $u_{x_2}$. Thus,
$$\begin{array}{l}
\displaystyle\int_{D_s}\big(u_{x_1}^2(x_1,x_2,x_3)+u_{x_2}^2(x_1,x_2,x_3)\big)\,dx_1\,dx_2\vspace{3pt}\\
\qquad\qquad\qquad\displaystyle\le 2\,\pi\,e^{2\gamma}+2\int_{D_{s+\beta\ln s}}\big(u_{x_1}^{2}(x_1,x_2,-A)+u_{x_2}^{2}(x_1,x_2,-A)\big)\,dx_1\,dx_2\end{array}$$
for all $s\ge1$ and $x_3<-A$, that is,
\[
K(s,x_3)\le2\pi e^{2\gamma}+2\,K(s+\beta\ln s,-A)\,,
\]
with $K$ defined in~\eqref{defK}. Notice that this last inequality also holds trivially with $x_3=-A$ since $s+\beta\ln s\ge s$. The proof of Lemma~\ref{K} is thereby complete with $C=2\pi e^{2\gamma}>0$.
\end{proof}

With a slight abuse of notation, we also write $u$ in the polar coordinates (in the $(x_1,x_2)$-plane) as $u(r,\theta,x_3)$. Let us now define, for $s>0$,
\begin{equation}\label{deff}
f\left(  s\right)  :=\int_{0}^{s}\int_{-\infty}^{-A}\int_{0}^{2\pi}u_{r}(r,\theta,x_3)\,%
u_{x_3}(r,\theta,x_3)\,d\theta\,dx_3\,dr.
\end{equation}
Note that the above integral converges absolutely for each $s>0$, from Remark~\ref{remexp}, and that $f(s)\to0$ as $s\to0$.

\begin{lemma}
\label{flimit}
The quantity $\alpha$ defined in~\eqref{defalpha} is such that $\alpha=0$.
\end{lemma}

\begin{proof}
Let us assume by way of contradiction that $\alpha>0$. Observe first that the function~$f$ is of class $C^1$ in $(0,+\infty)$ and that, thanks to~\eqref{integrals},
$$f'(s)=\frac{g(s)}{s}\ \hbox{ for all $s>0$}.$$
Using the fact that $g$ is nonnegative and $g(\tau)\to\alpha\in(0,+\infty]$ as $\tau\to+\infty$, we deduce that
\begin{equation}\label{fs}
f(s)=\int_0^sf^{\prime}(\tau)\,d\tau=\int_{0}^s\frac{g(\tau)}{\tau}\,d\tau
\geq\alpha'\ln s\ \hbox{ for all $s$ large enough},
\end{equation}
with, say, $\alpha'=\alpha/2>0$ if $0<\alpha<+\infty$ and $\alpha'=1$ if $\alpha=+\infty$. On the other hand, one infers Remark~\ref{remexp}, with here $a=0$ thanks to~\eqref{uH2}, that
\begin{equation}\label{nablau}
|\nabla u|\le 2\,\sqrt{2}\,e^{\sqrt{2}\,x_3}\ \hbox{ in }\R^3
\end{equation}
and from~\eqref{integrals} that
\begin{equation}\label{gtau}
\frac{g(\tau)}{\tau}=\int_{-\infty}^{-A}\int%
_{0}^{2\pi}u_{r}(\tau,\theta,x_3)\,u_{x_3}(\tau,\theta,x_3)\,d\theta\,dx_3\leq4\,\sqrt{2}\,\pi\le 6\,\pi
\end{equation}
for all $\tau>0$, hence
\begin{equation}\label{fs2}
f(s+\beta\ln s)-f(s)=\int_{s}^{s+\beta\ln s}\frac{g(\tau)}{\tau}\,d\tau\le6\,\pi\,\beta\,\ln s\ \hbox{ for all }s\ge1.
\end{equation}

Using again~\eqref{nablau} and~\eqref{gtau}, together with the definition~\eqref{defK} of $K$ and the decomposition of the integral~\eqref{deff} with respect to $r\in[0,s]$ into two integrals over $[1,s]$ and $[0,1]$, we get that, for all $s\ge1$,
\begin{align*}
f(s)  &  \leq \int_{-\infty}^{-A}\left[  \int_{1}^{s}\int%
_{0}^{2\pi}r\,u_{r}^{2}\,d\theta\,dr\right]  ^{1/2}\left[  \int_{1}^{s}\int%
_{0}^{2\pi}\frac{u_{x_3}^{2}}{r}\,d\theta\,dr\right]  ^{1/2}dx_3+6\,\pi\\
&  \le\sqrt{16\,\pi\ln s}\int_{-\infty}^{-A}\sqrt{K(s,x_3)}\,e^{\sqrt{2}\,x_3}\,dx_3+6\,\pi.
\end{align*}
Applying inequality~\eqref{gK} and Lemma~\ref{K}, we deduce that, for all $s\ge1$,
$$\begin{array}{rcl}
f(s)   &  \leq & \displaystyle\sqrt{16\,\pi\ln s}\int_{-\infty}^{-A}\sqrt{C(1+K(s+\beta\ln s,-A))}\,e^{\sqrt{2}\,x_3}\,dx_3+6\,\pi\vspace{3pt}\\
& \le & \sqrt{8\,\pi\,C\,(\ln s)\,\big(1+2g(s+\beta\ln s)\big)}+6\,\pi\vspace{3pt}\\
& = & \sqrt{8\,\pi\,C\,(\ln s)\,\big(1+2f'(s+\beta\ln s)\,(s+\beta\ln s)\big)}+6\,\pi.\end{array}$$
Together with~\eqref{fs}, it follows that $f'(t)\,t\to+\infty$ as $t\to+\infty$, and that
$$0<f(s)\le\sqrt{17\,\pi\,C\,(s+\beta\ln s)\,\ln(s+\beta\ln s)\,f'(s+\beta\ln s)}\ \hbox{ for all $s$ large enough}.$$
Thanks to~\eqref{fs} and~\eqref{fs2}, we then infer that, for all $s$ large enough,
$$\begin{array}{rcl}
0<f(s+\beta\ln s)\leq f(s)+6\,\pi\,\beta\,\ln s & \!\!\le\!\! & \displaystyle\Big(1+\frac{6\,\pi\,\beta}{\alpha'}\Big)\,f(s)\vspace{3pt}\\
& \!\!\le\!\! & C_1\,\sqrt{(s+\beta\ln s)\,\ln(s+\beta\ln s)\,f'(s+\beta\ln s)}
\end{array}$$
with
\[
C_1=\sqrt{17\,\pi\,C}\,\left(1+ \frac{6\,\pi\,\beta}{\alpha'}\right)>0\,.
\]
In other words, there is $t_{0}>0$ such that $f>0$ on~$[t_0,+\infty)$ and%
\[
\frac{f^{\prime}(t)}{f(t)^{2}}\geq\frac{1}{C_1^2\,t\ln
t}\text{ for all }t\geq t_{0}.\text{ }%
\]
It follows that the function
\[
t\mapsto\frac{1}{f(t)}+\frac{\ln(\ln t)}{C_1^2}
\]
is nonincreasing on $[t_0,+\infty)$. But since $f>0$ on $[t_0,+\infty)$, one has
\[
\frac{1}{f(t)}+ \frac{\ln(\ln t)}{C_1^2}\to+\infty
\]
as $t\to+\infty$. This is a contradiction. Therefore, $\alpha=0$ and the proof of Lemma~\ref{flimit} is thereby complete.
\end{proof}

\begin{proof}[End of the proof of Theorem~$\ref{main}$ for $n=3$]
As in the case $n=2$, the fact that $\alpha=0$ in~\eqref{integrals} and~\eqref{defalpha}, together with~\eqref{I}, implies that
$$\begin{array}{rcl}
\displaystyle\frac{|\nabla u(x_1,x_2,-A)|^2}{2}+F(u(x_1,x_2,-A))-u_{x_3}^2(x_1,x_2,-A) & & \vspace{3pt}\\
\qquad\qquad\qquad\displaystyle=\ \frac{u_{x_1}^2(x_1,x_2,-A)+u_{x_2}^2(x_1,x_2,-A)}{2} & = & 0\end{array}$$
for all $(x_1,x_2)\in\R^2$. Hence
\[
F(u(x_1,x_2,-A))=\frac{u_{x_3}^2(x_1,x_2,-A)}{2}=\frac{|\nabla u(x_1,x_2,-A)|^2}{2}
\]
for all $(x_1,x_2)\in\R^2$ and one concludes from~\cite{CGS,Modica} that $u$ is one-dimensional, namely $u(x_1,x_2,x_3)\equiv H(x_3)$ in $\R^3$ from the second paragraph of this subsection.
\end{proof}

\begin{remark}{\rm In higher dimensions $n\ge4$, one could still apply the balancing condition and define some functions $g$ and $f$ with formulas similar to~\eqref{integrals} and~\eqref{deff} above. However, in~\eqref{integrals}, there would be a factor $s^{n-2}$ instead of $s$ in the right-hand side. Even if Lemma~$\ref{K}$ still extends to that case (with a different value for the constant $\beta$), Lemma~$\ref{flimit}$ does not extend as such. In particular, one would have $f'(s)=g(s)/s^{n-2}$, and the integrability of the function $1/s^{n-2}$ at infinity does not imply that $f(+\infty)=+\infty$ if~$\alpha:=g(+\infty)>0$, and then the end of the proof does not work.}
\end{remark}


\section{Half-space theorems for free boundary problems}\label{free}

In this section, we are interested in half-space properties for free boundary
problems. First of all, we consider the following classical one phase free
boundary problem:%
\begin{equation}
\left\{
\begin{array}
[c]{rcl}%
\Delta u & = & 0\ \text{ in }\Xi:=\left\{  u>0\right\}  \subset\mathbb{R}^{n},\vspace{3pt}\\
\left\vert \nabla u\right\vert & = & 1\ \text{ on }\partial\Xi,
\end{array}
\right.  \label{one-phase}%
\end{equation}
where $u$ is understood in the classical sense in $\overline{\Xi}$ and $\partial\Xi$ is globally smooth.

The existence of catenoid type solutions of this problem has been proved using an
Allen-Cahn approximation. We refer to~\cite{LWW2} and the references therein
for more discussion on this problem.

We have the following half-space property:

\begin{theorem}
\label{one}Let $n\leq3$ and $u$ be a solution of~\eqref{one-phase}
with $\left\vert \nabla u\right\vert \leq1.$ Suppose that the
positive phase $\Xi$ is contained in the half-space $\left\{  x_{n}>0\right\}
$. Then $u$ is one-dimensional, namely there is $h\ge0$ such that $\Xi=\{x_n>h\}$ and $u$ is the one-dimensional function $u(x)\equiv x_n-h$ in~$\overline{\Xi}$.
\end{theorem}

\begin{proof}
The idea of proof is same as that of Section~\ref{n=3}. We sketch the proof
and list the necessary modifications. Let us only consider the case $n=3.$

Up to shift in the $x_3$-direction, one can assume without loss of generality that $\Xi$ is
not contained in $\left\{  x_3>a\right\}  $ for any $a>0.$ From standard elliptic estimates up to the boundary, one can fix $a>0$ small enough such that $u_{x_3}>0$ in $\overline{\Xi}\cap\{  x_3\le a\}$.

We still adopt the notation
of Section~\ref{n=3} and, for $s>0$, let $\Omega_s$ be the half-cylinder
\[
\Omega_s:=\left\{  \left(  x_1,x_2,x_3\right)\in\R^3  :r^2=x_1^{2}+x_2^{2}<s^2,\ \ x_3<a\right\}  .
\]
For $\varepsilon>0$, let us define
\[
\Xi_{\varepsilon}:=\left\{  Z\in\Xi : {\rm{dist}}\left(  Z,\partial\Xi\right)
>\varepsilon\right\}  .
\]
Let $F$ be half the characteristic function of the interval $\left(  0,+\infty\right)$,
that is, $F(\tau)=1/2$ if $\tau>0$ and $F(\tau)=0$ if $\tau\le0$.
Then we have the following balancing formula, with $X=(0,0,1)$ and $\varepsilon\in(0,a)$:%
\[
\int_{\partial\left(  \Omega_s\cap\Xi_{\varepsilon}\right)  }\left[  \left(
\frac{1}{2}\left\vert \nabla u\right\vert ^{2}+F\left(  u\right)  \right)
X\cdot\nu-\left(  \nabla u\cdot X\right)  \left(  \nabla u\cdot\nu\right)
\right]  d\sigma=0.
\]
Sending $\varepsilon$ to $0$ in this identity and using the free boundary
condition, we get
\[
\int_{\partial\Omega_s\cap\Xi}\left[  \left(  \frac{1}%
{2}\left\vert \nabla u\right\vert ^{2}+F\left(  u\right)  \right)  X\cdot
\nu-\left(  \nabla u\cdot X\right)  \left(  \nabla u\cdot\nu\right)  \right]
d\sigma=0.
\]
Now we extend the solution $u$ to $\mathbb{R}^{3}$ such that $u=0$ in
$\mathbb{R}^{3}\backslash\Xi.$ Still denote it as $u.$ Then we get
\[
\int_{\partial\Omega_s}\left[  \left(  \frac{1}{2}\left\vert \nabla u\right\vert
^{2}+F\left(  u\right)  \right)  X\cdot\nu-\left(  \nabla u\cdot X\right)
\left(  \nabla u\cdot\nu\right)  \right]  d\sigma=0.
\]
Note that $u$ is not smooth across the free boundary, but, for any $Y\in\partial\Xi$, the quantity $(|\nabla u(Z)|^2/2+F(u(Z)))\,X\cdot\nu(Y)-(\nabla u(Z)\cdot X)\,(\nabla u(Z)\cdot\nu(Y))$ converges to $0$ as $Z\to Y$ with $Z\in\Xi$ and it vanishes for all $Z\in\R^3\setminus\overline{\Xi}$.

With the same slight abuse of notation as in Section~\ref{n=3}, we define, for $s>0$,
\[
f\left(  s\right)  =\int_{0}^{s}\int_{-\infty}^{a}\int_{0}^{2\pi}u_{r}(r,\theta,x_3)%
u_{x_3}(r,\theta,x_3)\,d\theta\,dx_3\,dr.
\]
Since $|\nabla u|\le1$ in $\Xi$, one has $|\nabla u|^2/2+1/2\ge u_{x_3}^2$ in $\Xi$,
hence the function $f$ is nonnegative, non-decreasing and differentiable with
respect to $s.$ Similarly to the proof of Section~\ref{n=3}, we can show that, if
$$\lim_{s\to+\infty}\int_{\{x_1^2+x_2^2<s^2\}}\left(\frac12|\nabla u|^2+F(u)-u_{x_3}^2\right)\,dx_1\,dx_2>0,$$
then there is a positive constant $C>0$ such that
\[
f^{\prime}\left(  s\right)  \geq\frac{Cf^{2}\left(  s\right)  }{s\ln s},\text{
for }s\text{ large.}%
\]
The previous inequality yields a contradiction as in Section~\ref{n=3}. This then implies that
$\left\vert \nabla
u(x_1,x_2,a)\right\vert =1$ and $u_{x_1}(x_1,x_2,a)=u_{x_2}(x_1,x_2,a)=0$ for all $(x_1,x_2)\in\R^2$ such that $(x_1,x_2,a)\in\Xi$. Since $\Delta\left(  \left\vert \nabla u\right\vert
^{2}\right)=2\sum_{1\le i,j\le 3}u_{x_ix_j}^2\geq0$ in $\Xi$, we conclude that $|\nabla u|=1$ and $u_{x_1}=u_{x_2}=0$ in each connected component of $\Xi$ meeting $\{x_3=a\}$. It finally follows that $\Xi\supset\{x_3=a_k\}$, for a sequence $(a_k)_{k\in\N}$ with $a_k\to0^+$ (from the normalization made in the second paragraph of the proof) and, remembering that $u_{x_3}>0$ in $\overline{\Xi}\cap\{x_3\le a\}$, we easily conclude that $\Xi=\{x_3>0\}$ and $u$ is the one-dimensional function $u(x_1,x_2,x_3)\equiv x_3$.
\end{proof}

Similarly, we can consider the following double-well type free boundary
problem:%
\begin{equation}
\left\{
\begin{array}
[c]{rcl}%
\Delta u & = & 0\ \text{ in }\Xi:=\left\{  \left\vert u\right\vert <1\right\}
\subset\mathbb{R}^{n},\vspace{3pt}\\
\left\vert \nabla u\right\vert & = & 1\ \text{ on }\partial\Xi.
\end{array}
\right.  \label{two-phase}%
\end{equation}
The proof of the following result is essentially same as that of Theorem
\ref{one}, and we omit the details.

\begin{theorem}
\label{two}Let $n\leq3$ and $u$ be a solution of~\eqref{two-phase}
with $\left\vert \nabla u\right\vert \leq1.$ Suppose that $\left\{
\left\vert u\right\vert <1\right\}$ is contained in the half-space $\left\{
x_{n}>0\right\}  $. Then $u$ is one-dimensional, namely there is $h\ge1$ such that $\Xi=\{h-1<x_n<h+1\}$, and either $u(x)\equiv x_n-h$ in $\overline{\Xi}$ or $u(x)\equiv-(x_n-h)$ in $\overline{\Xi}$.
\end{theorem}

\noindent\textit{Acknowledgements.} {\small{F. Hamel is partially supported by: the Excellence Initiative of Aix-Marseille University~-~A*MIDEX, a French ``Investissements d'Avenir'' programme, the European Research Council under the European Union's Seventh Framework Programme (FP/2007-2013) ERC Grant Agreement n.~321186~- ReaDi~- Reaction-Diffusion Equations, Propagation and Model\-ling, and the ANR NONLOCAL project (ANR-14-CE25-0013). Y. Liu is partially supported by \textquotedblleft The Fundamental Research Funds for the Central Universities WK3470000014\textquotedblright. P. Sicbaldi is partially supported by the grant \textquotedblleft Ram\'on y Cajal 2015\textquotedblright\, RYC-2015-18730 and the grant \textquotedblleft Ana\-lisis geom\'etrico\textquotedblright\, MTM 2017-89677-P.
K.~Wang is supported by NSFC no. 11871381. J.~Wei is partially supported by NSERC of Canada.\\
Part of the paper was finished while Y.~ÊLiu was visiting the University of British Columbia in 2019, and he appreciates the institution for its hospitality and financial support. Part of this work was also completed while F.~Hamel and J.~Wei were visiting the University of Granada in 2019 in occasion of the conference \textquotedblleft Geometry and PDE in front of the Alhambra\textquotedblright, and they also appreciate the institution for the hospitality and the financial support. Finally, part of this work has been carried out in the framework of Archim\`ede Labex of Aix-Marseille University.\\
The authors are also grateful to Alberto Farina for pointing out his results~\cite{Farina98,Farina03} and their relation with our paper.}}



\begin{thebibliography}{99}

\bibitem {ADW}O. Agudelo, M. Del Pino, and J. Wei.
Higher-dimensional catenoid, Liouville equation, and Allen-Cahn equation,
\emph{Int. Math. Res. Not.} 23 (2016), 7051--7102.

\bibitem {AC}G. Alberti, L. Ambrosio, and X. Cabr\'e. On a long-standing conjecture
of E. De Giorgi: symmetry in 3D for general nonlinearities and a local
minimality property, \emph{Acta Appl. Math.} 65 (2001), 9--33.

\bibitem {C}L. Ambrosio and X. Cabr\'e. Entire solutions of semilinear elliptic
equations in $\mathbb{R}^3$ and a conjecture of De Giorgi, \emph{J. Amer. Math. Soc.} 13
(2000), 725--739.

\bibitem{BCN} H. Berestycki, L. A. Caffarelli, and L. Nirenberg. Further qualitative properties for elliptic equations in unbounded domains, \emph{Ann. Scuola Norm. Sup. Pisa Cl. Sci. (4)} 5 (1997), 69--94.

\bibitem{BCN2} H. Berestycki, L. A. Caffarelli, and L. Nirenberg. Monotonicity for elliptic equations in unbounded Lipschitz domains, \emph{Comm. Pure Appl. Math.} 50 (1997), 1089--1111.

\bibitem{BH} H. Berestycki and F. Hamel. Generalized travelling waves for reaction-diffusion equations, In: \emph{Perspectives in Nonlinear Partial Differential Equations. In honor of H.~Brezis}, Amer. Math. Soc., Contemp. Math., 2007, 101-123.

\bibitem {C3}X. Cabr\'e. Uniqueness and stability of saddle-shaped solutions to
the Allen-Cahn equation, \emph{J. Math. Pures Appl.} 98 (2012), 239--256.

\bibitem{CGS} L. A. Caffarelli, N. Garofalo, and F. Segala. A gradient bound for entire solutions of quasi-linear equations and its consequences, \emph{Comm. Pure Appl. Math.} 47 (1994), 1457--1473.

\bibitem {CM}T. H. Colding and W. P. Minicozzi II. The Calabi-Yau conjectures
for embedded surfaces, \emph{Ann. Math. (2)} 167 (2008), 211--243.

\bibitem {DeG}E. De Giorgi. Convergence problems for functionals and
operators, \emph{Proc. Int. Meeting on Recent Methods in Nonlinear Analysis (Rome,
1978)}, Pitagora, Bologna (1979).

\bibitem {DeS-S}D. De Silva and O. Savin. Symmetry of global solutions to a
class of fully nonlinear elliptic equations in 2D, \emph{Indiana Univ.
Math. J.} 58 (2009), 301--315.

\bibitem {DKP}M. Del Pino, M. Kowalczyk, and F. Pacard. Moduli space theory
for the Allen-Cahn equation in the plane, \emph{Trans. Amer. Math. Soc.} 365
(2013), 721--766.

\bibitem {M2}M. Del Pino, M. Kowalczyk, and J. Wei. On De Giorgi's conjecture in
dimension $N\geq9$, \emph{Ann. Math. (2)} 174 (2011), 1485--1569.

\bibitem{Farina98} A. Farina. Finite-energy solutions, quantization effects and Liouville-type results for a variant of the Ginzburg- Landau systems in $\R^K$, \emph{Diff. Int. Equations} 11 (1998), 875--893.

\bibitem {Farina0}A. Farina. Symmetry for solutions of semilinear elliptic
equations in $\mathbb{R}^{n}$ and related conjectures, \emph{Ricerche Mat.} 48
(1999), 129--154.

\bibitem{Farina03} A. Farina. Rigidity and one-dimensional symmetry for semilinear elliptic equations in the whole of $\R^N$ and in half spaces, \emph{Adv. Math. Sciences Appl.} 13 (2003), 65--82. 

\bibitem {Farina2}A. Farina, B. Sciunzi, and E. Valdinoci. Bernstein and De Giorgi
type problems: new results via a geometric approach, \emph{Ann. Scuola Norm.
Sup. Pisa Cl. Sci. (5)} VII (2008), 741--791

\bibitem {Farina}A. Farina and E. Valdinoci. The state of the art for a
conjecture of De Giorgi and related problems, in: \emph{Recent Progress on
Reaction--Diffusion Systems and Viscosity Solutions}, World Sci. Publ.,
Hackensack, NJ, 2009, 74--96.

\bibitem {Farina3}A. Farina and E. Valdinoci. Flattening results for elliptic
PDEs in unbounded domains with applications to overdetermined problems,
\emph{Arch. Ration. Mech. Anal.} 195 (2010), 1025--1058.

\bibitem {G}N. Ghoussoub and C. Gui. On a conjecture of De Giorgi and some
related problems, \emph{Math. Ann.} 311 (1998), 481--491.

\bibitem{GNN} B. Gidas, W.-M. Ni, and L. Nirenberg. Symmetry and related properties via the maximum principle, \emph{Comm. Math. Phys.} 68 (1979), 209--243.

\bibitem {HM}D. Hoffman and W. H. Meeks, The strong halfspace theorem for
minimal surfaces, \emph{Invent. Math.} 101 (1990), 373--377.

\bibitem {LWW}Y. Liu, K. Wang, and J. Wei. Global minimizers of Allen-Cahn
equation in dimensions $n\geq8$, \emph{J. Math. Pures Appl.} 108 (2017), 818--840.

\bibitem {LWW2}Y. Liu, K. Wang, and J. Wei. On smooth solutions to one phase
free boundary problem in $\mathbb{R}^{n}$, preprint.

\bibitem {MPR}W. H. Meeks, J. Perez, and A. Ros. Local removable singularity
theorems for minimal laminations, \emph{J. Diff. Geom.} 103 (2016), 319--392.

\bibitem {MR}W. H. Meeks and H. Rosenberg. Maximum principles at infinity,
\emph{J. Diff. Geom.} 79 (2008), 141--165.

\bibitem {Modica}L. Modica. A gradient bound and a Liouville theorem for
nonlinear Poisson equations, \emph{Comm. Pure Appl. Math.} 38 (1985), 679--684.

\bibitem {S}O. Savin. Regularity of flat level sets in phase transitions.
\emph{Ann. Math. (2)} 169 (2009), 41--78.

\bibitem {X}F. Xavier. Convex hulls of complete minimal surfaces, \emph{Math.
Ann.} 269 (1984) 179--182.

\end{thebibliography}
\end{document}